\documentclass[reqno, 11pt]{amsart}
\usepackage{amsmath}	
\usepackage{amssymb}	
\usepackage{amsthm}
\usepackage{amsfonts}
\usepackage{latexsym}
\usepackage{mathrsfs}
\usepackage[all,ps,cmtip]{xy}
\usepackage{tabularx}
\usepackage{pict2e}
\usepackage{graphicx}
\usepackage{xypic}
\usepackage{tikz-cd}
\usepackage{tikz}
\usetikzlibrary{matrix,arrows}
\usepackage[colorlinks=true, citecolor=blue, urlcolor=blue, linkcolor=blue]{hyperref}
\usepackage[text={145mm, 230mm},centering]{geometry}
\input diagxy
\input xy

\theoremstyle{plain} 
\newtheorem{thm}{Theorem}[section]
\newtheorem{lem}[thm]{Lemma} 
\newtheorem{prop}[thm]{Proposition} 
\newtheorem{cor}[thm]{Corollary} 

\theoremstyle{definition} 
\newtheorem{defn}[thm]{Definition}
\newtheorem{rem}[thm]{Remark} 

\newtheorem{quest}[thm]{Question}
\newtheorem{conj}[thm]{Conjecture}

\DeclareMathOperator{\Z}{\mathbb{Z}}

\DeclareMathOperator{\CO}{\mathcal{O}}
\DeclareMathOperator{\Ku}{\mathrm{Ku}}

\allowdisplaybreaks

\numberwithin{equation}{section}

\begin{document}

\title[]{Rational cubic fourfolds in Hassett divisors}

\author{Song Yang and Xun Yu}
\address{Center for Applied Mathematics, Tianjin University, Weijin Road 92, Tianjin 300072, China}%

\email{syangmath@tju.edu.cn, xunyu@tju.edu.cn}%

\date{\today}

\begin{abstract}
We prove that every Hassett's Noether-Lefschetz divisor of special cubic fourfolds contains a union of three codimension-two subvarieties, parametrizing rational cubic fourfolds, in the moduli space of smooth cubic fourfolds. 
\end{abstract}

\subjclass[2010]{14C30, 14E08, 14M20}
\keywords{Hodge theory, special cubic fourfolds, lattice}

\maketitle


\section{Introduction}

The rationality problem of smooth cubic fourfolds is one of the most widely open problems in algebraic geometry; 
we refer to the survey \cite{Has16} for a comprehensive progress.
It has been known that all smooth cubic surfaces are rational since the 19th century.
In 1972, Clemens-Griffiths \cite{CG72} proved that 
all smooth cubic threefolds are nonrational. 
For smooth cubic fourfolds, however, the situation is very mysterious.
It is expected that a very general smooth cubic fourfold should be nonrational (cf. \cite{Has99, Has00}). 
Until now, many examples of smooth rational cubic fourfolds are known, 
but the existence of a smooth nonrational cubic fourfold is still unknown.

Using Hodge theory and lattice theory, 
Hassett \cite{Has00} introduced the notion of {\it special cubic fourfolds} (see Definition \ref{specialcubic}).
Simultaneously, Hassett \cite[Theorem 1.0.1]{Has00} gave a countably infinite 
irreducible divisors $\mathcal{C}_{d}$ of special cubic fourfolds 
in the moduli space $\mathcal{C}$ of smooth cubic fourfolds and showed that $\mathcal{C}_{d}$ is nonempty if and only if 
$d>6$ and  $d\equiv 0, 2\; (\mathrm{mod}\, 6)$. 
Such a nonempty $\mathcal{C}_{d}$ 
is called a {\it Hassett's Noether-Lefschetz divisor} (for short a {\it Hassett divisor}).

Currently, there exist two popular point of views toward the rationality of smooth cubic fourfolds and both have associated $K3$ surfaces:
\begin{itemize}
\item Hassett's Hodge-theoretic result (\cite[Theorem 5.1.3]{Has00}): 
        a smooth cubic fourfold $X$ has a Hodge-theoretically associated $K3$ surface if and only if $X\in \mathcal{C}_{d}$ for some {\it admissible value} $d$ (i.e., $d>6$, $d\equiv 0, 2\; (\mathrm{mod}\, 6)$, $4 \nmid d$, $9 \nmid d$ and $p \nmid d$ for any odd prime $p \equiv 2\;  (\mathrm{mod}\, 3)$);

\item Kuznetsov's derived categorical conjecture (\cite[Conjecture 1.1]{Kuz10}): a smooth cubic fourfold $X$ is rational if and only if its Kuznetsov component $\Ku(X)$ is derived equivalent to a $K3$ surface (i.e., $\Ku(X)$ is called {\it geometric}), where $\Ku(X)$ is the right orthogonal to $\{\CO_{X}, \CO_{X}(1), \CO_{X}(2)\}$.
\end{itemize}

It is important to notice that Kuznetsov's conjecture implies that a very general cubic fourfold is not rational, since for a very general cubic fourfold its Kuznetsov component can not be geometric.
Addinton-Thomas \cite[Theorem 1.1]{AT14} showed that for a smooth cubic fourfold $X$ if $\Ku(X)$ is geometric then $X\in \mathcal{C}_{d}$ for some admissible $d$, and conversely for any admissible value $d$,  the set of cubic fourfolds $X\in  \mathcal{C}_{d}$ for which $\Ku(X)$ is geometric is a Zariski open dense subset; see also Huybrechts \cite{Huy17} for the twisted version and a further study. 
Recently, based on Bridgeland stability conditions on $\Ku(X)$ constructed in \cite[Theorem 1.2]{BLMS17}, Bayer-Lahoz-Macr\`{i}-Nuer-Perry-Stellari \cite[Corollary 29.7]{BLMNPS19} proved that for any admissible value $d$, $\Ku(X)$ of {\it every} $X\in \mathcal{C}_{d}$ is geometric. So we now know that for a smooth cubic fourfold $X$ its Kuznetsov component $\Ku(X)$ is geometric if and only if $X\in \mathcal{C}_{d}$ for some admissible value $d$.
Then one can restate Kuznetsov's conjecture as the following equivalent form.

\begin{conj}\label{mainconj}
A smooth cubic fourfold $X$ is {\it rational} if and only if $X\in \mathcal{C}_{d}$ for some admissible value $d$.
\end{conj}

 The first three admissible values are $14, 26, 38$. Every cubic fourfold in $\mathcal{C}_{14}$ is rational \cite{Fan43, BRS19}; see also \cite[Theorem 2]{RS19} for a different proof. 
Based on Kontsevich-Tschinkel \cite[Theorem 1]{KT19}, 
Russo-Staglian\`{o} \cite[Theorems 4, 7]{RS19} finally showed that every cubic fourfold in $\mathcal{C}_{26}$ and $\mathcal{C}_{38}$ is rational; 
see also \cite{RS18} for the construction of explicit birational maps. 
So far ``if '' part of Conjecture \ref{mainconj} has been confirmed only for the three Hassett divisors $\mathcal{C}_{14},\mathcal{C}_{26},\mathcal{C}_{38}$. Thus finding rational cubic fourfolds in other Hassett divisors is of interest. The main result of this paper is the following.

\begin{thm}[=Theorem \ref{mainthm1}]\label{mainthm}
Every Hassett divisor $\mathcal{C}_{d}$ contains a union of three codimension-two subvarieties  in $\mathcal{C}$ parametrizing rational cubic fourfolds.
\end{thm}

The idea of the proof is simple: 
we first show any two Hassett divisors intersect by Theorem \ref{mainprop}, 
which is of independent interest (for considerations of the intersections among Hassett divisors, see \cite{Has99, AT14, ABBVA14, BRS19} etc.), 
and finally we consider the intersections $\mathcal{C}_{d}\cap\mathcal{C}_{14}$, $\mathcal{C}_{d}\cap\mathcal{C}_{26}$ and $\mathcal{C}_{d}\cap\mathcal{C}_{38}$ for every Hassett divisor $\mathcal{C}_{d}$. 

Throughout this paper, we work over the complex number field $\mathbb{C}$.

\subsection*{Acknowledgements}
We would like to thank Professors Keiji Oguiso and Paolo Stellari for useful conversations.
This work is partially supported by the National Natural Science Foundation of China (Grant No. 11701413, No. 11701414).



\section{Lattice and Hodge theory for cubic fourfolds}

In this section, we collect some known results on Hodge structures and lattices associated with smooth cubic fourfolds. We refer to \cite{BD85, Has00, Has16, Huy18} for more detailed discussions, 
specially for the Hodge-theoretic aspect, 
and to \cite{Ser73, Nik80} for the basics of abstract lattice theory.

The cubic hypersurfaces in $\mathbb{P}^{5}$ are parametrized by  
$\mathbb{P}(H^{0}(\mathbb{P}^{5}, \mathcal{O}(3)))\cong \mathbb{P}^{55}$.
Moreover, the smooth cubic hypersurfaces form a Zariski open dense subset $\mathcal{U}\subset \mathbb{P}^{55}$.
Then the moduli space of smooth cubic fourfolds is the quotient space
$$
\mathcal{C}:=\mathcal{U}// \mathrm{PGL(6, \mathbb{C})}
$$
which is a $20$-dimensional quasi-projective variety.

Let $X$ be a smooth cubic fourfold.
Then the cohomology $H^{\ast}(X, \Z)$ is torsion-free and the middle Hodge cohomology of $X$ is as follows:
\begin{equation*}
\begin{array}{cccc}
\begin{matrix}
&0 && 1 && 21 && 1 && 0.
\end{matrix}
\end{array}
\end{equation*}
The Hodge-Riemann bilinear relations imply that $H^{4}(X, \Z)$ is a unimodular lattice under the intersection form $(\,.\,)$ of signature $(21,2)$.
Furthermore, as abstract lattices, 
\cite[Proposition 2.1.2]{Has00} implies the middle cohomology and the primitive cohomology 
$$
L:=H^{4}(X, \Z)
\simeq E_{8}^{\oplus 2}\oplus U^{\oplus 2}\oplus I_{3,0}
$$
$$
L^{0}:=(h^{2})^{\perp}=H_{\mathrm{prim}}^{4}(X, \Z)
\simeq E_{8}^{\oplus 2}\oplus U^{\oplus 2}\oplus A_{2}
$$
where the square of the hyperplane class $h$ 
is given as $h^{2}=(1, 1, 1)\in I_{3,0}$ of which the intersection form is given by the identity matrix of rank $3$,
$
A_{2}=\small{\left(\begin{array}{cc} 
2 & 1 \\
1 & 2 
\end{array} \right)},
$
$
U=\small{\left(\begin{array}{cc} 
0 & 1 \\
1 & 0 
\end{array} \right)}
$
the hyperbolic plane, $E_{8}$ is the unimodular positive definite even lattice of rank $8$.
Note that $L^{0}$ is an even lattice.

\begin{defn}[Hassett \cite{Has00}]\label{specialcubic}
A smooth cubic fourfold $X$ is called {\it special} if it contains an algebraic surface not homologous to a complete intersection. 
\end{defn}

The integral Hodge conjecture holds for smooth cubic fourfolds (\cite[Theorem 18]{Voi07} or see also \cite[Corollary 29.8]{BLMNPS19} for a new proof). Thus, a smooth cubic fourfold $X$ is {\it special} if and only if the rank of the positive definite lattice
$$
A(X):=H^{4}(X, \Z)\cap H^{2,2}(X)
$$
is at least $2$.

\begin{defn}[Hassett \cite{Has00}]
A {\it labelling} of a special cubic fourfold consists of a positive definite rank two saturated (i.e. the quotient group $A(X)/K$ is torsion free) sublattice
$$
h^{2}\in K\subset A(X),
$$
and its discriminant $d$ is the determinant of the intersection form on $K$.
\end{defn}

In \cite{Has00}, Hassett defined $\mathcal{C}_{d}$ as the set of special cubic fourfolds $X$ with labelling of discriminant $d$. 
Moreover, Hassett \cite[Theorem 1.0.1]{Has00} showed that $\mathcal{C}_{d}\subset \mathcal{C}$ is an irreducible divisor and is nonempty if and only if 
\begin{equation}\tag{$\star$}
d>6 \;\;\; \textrm{and}\;\;\; d\equiv 0, 2\; (\mathrm{mod}\, 6).
\end{equation}

The following proposition is a generalization of \cite[Theorems 1.0.1]{Has00}.

\begin{prop}[{\cite[Proposition 12 and page 43]{Has16}}]\label{crucial}

Fix a positive definite lattice $M$ of rank $r\ge 2$ admitting a saturated embedding 

$$h^2\in M\subset L.$$  We denote by $\mathcal{C}_M\subset \mathcal{C}$ the smooth cubic fourfolds $X$ admitting algebraic classes with this lattice structure $$h^2\in M\subset A(X)\subset L.$$ Then $\mathcal{C}_M$ has codimension $r-1$ and there exists $X\in \mathcal{C}_M$ with $A(X)=M$, provided $\mathcal{C}_M$ is nonempty. Moreover, $\mathcal{C}_M$ is nonempty if and only if there exists no sublattice $h^{2}\in K\subset M$ with $K=K_{2}$ or $K_{6}$, where $
K_{2}=\small{\left(\begin{array}{cc} 
3 & 1 \\
1 & 1 
\end{array} \right)}
$ 
and
$
K_{6}=\small{\left(\begin{array}{cc} 
3 & 0 \\
0 & 2 
\end{array} \right)}
$.
\end{prop}

This proposition is crucial for our purpose,
so we sketch a proof for the convenience of readers.

\begin{proof}[Sketch of proof]
Suppose $\mathcal{C}_{M}$ is nonempty.
If $h^{2}\in K_{6}\subset M$ is a sublattice,
then there is a smooth cubic fourfold $X$ 
such that $A(X)\cap \langle h^{2}\rangle^{\perp}$ contains an element $r$ with $(r.r)=2$ and this contradicts Voisin \cite[Section 4, Proposition 1]{Voi85}; 
furthermore, Hassett \cite[Theorem 4.4.1]{Has00} excludes the case when $h^{2}\in K_{2}\subset M$ is a sublattice. 

Conversely, suppose that there exists no rank two sublattice $h^{2}\in K\subset M$ with $K=K_{2}$ or $K_{6}$. Since the signature of $L$ is $(21, 2)$ and $M\subset L$ is positive definite, 
by a standard argument, 
one can always find $\omega\in L\otimes_{\Z} \mathbb{C}$ such that
$$
\omega^{2}=0,\;\;  (\omega.\bar{\omega})<0\;\; \textrm{and} \;\; L\cap \omega^{\perp}=M.
$$
According to the description of the image of the period map for cubic fourfolds (Laza \cite[Theorem 1.1]{Laz10} and Looijenga \cite[Theorem 4.1]{Loo09}), 
one has a smooth cubic fourfold $X$ 
and a complete marking $\phi: H^{4}(X, \Z)\stackrel{\simeq}\longrightarrow L$ map the square of the hyperplane class to $h^{2}\in L$ and a generator of $H^{3,1}(X)$ to $\omega$.
Thus $M=A(X)$ and hence $\mathcal{C}_{M}$ contains $X$ and nonempty.
\end{proof}

In the rest of the context, 
we will frequently use the following lemma to check the nonemptyness condition in the Proposition \ref{crucial}.

\begin{lem}\label{key-observation}
Let $h^{2}\in M\subset L$ be a positive definite saturated sublattice.
Then the following three conditions are equivalent:
\begin{enumerate}
\item[(i)] there exists no sublattice $h^{2}\in K\subset M$ with $K=K_{2}$ or $K_{6}$;
\item[(ii)] there exists no $r\in M$ such that $(r.r)=2$ (i.e., $M$ does not represent $2$);
\item[(iii)] for any $0\neq x\in M$, $(x.x)\geq 3$.
\end{enumerate}
In particular, if $M$ satisfies one of the three equivalent conditions, 
then $\emptyset\neq \mathcal{C}_{M}\subset \mathcal{C}_{K}$ 
for any saturated sublattice $h^{2}\in K\subset M$.
\end{lem}

\begin{proof}

First of all,  $(ii) \Rightarrow (i)$ is clear since both $K_{2}$ and $K_{6}$ represent $2$.

Secondly, $(i) \Rightarrow (ii)$.
Suppose that there exists $r\in M$ such that $(r.r)=2$. 
We denote by $K\subset M$ the sublattice generated by $h^{2}$ and $r$.
Hence, the Gram matrix of $K$ with respect to the basis $(h, r)$ is
$$
\left(\begin{array}{cc} 
(h^{2}.h^{2}) & (h^{2}.r) \\
(h^{2}.r) & (r.r) 
\end{array} \right)
=
\left(\begin{array}{cc} 
3 & a \\
a & 2 
\end{array} \right).
$$
Replacing $r$ by $-r$ if necessary, we may and will assume $a\ge 0$.
Since $K$ is positive definite, we have $a^{2}<6$ and thus $a=0, 1, 2$. If $a=0$ (resp. $2$) , then $K$ is isometric to $K_{6}$ (resp. $K_{2}$), contradiction. If $a=1$, then $h^2-3r\in (h^{2})^{\perp}=L^0$ and $((h^2-3r).(h^2-3r))=15$, an odd number, contradicting to the fact $L^0$ is even.  

Finally, clearly $(iii)$ implies $(ii)$.
Conversely, we show $(ii)$ implies $(iii)$. 
By hypothesis, we may assume that there is $r\in M$ with $(r.r)=1$. 
Then let $K\subset M$ be the sublattice generated by $h^2$ and $r$. Hence, the Gram matrix of $K$ with respect to the basis $(h, r)$ is
$$
\left(\begin{array}{cc} 
3 & a \\
a & 1 
\end{array} \right)
$$
where $a=(h^{2}.r)$. Replacing $r$ by $-r$ if necessary, we may and will assume $a\ge 0$.
Since $K$ is positive definite, we have $a^{2}<3$ and thus $a=0, 1$. If $a=0$, then $r\in (h^{2})^{\perp}=L^0$ and $(r.r)=1$, an odd number, contradicting to the fact $L^0$ is even.  If $a=1$, then $K$ is isometric to $K_2$ and $K$ represents $2$, contradiction.
\end{proof}


\section{Intersections of Hassett divisors}

In this section, 
we prove Theorem \ref{mainthm} (=Theorem \ref{mainthm1})
and discuss some related results (Theorem \ref{mainprop} and Theorem \ref{mainthm2}).

Firstly, we setup some notations for latter use. Let

$$
L= E_{8}^{\oplus 2}\oplus U_{1}\oplus U_{2} \oplus I_{3,0},
$$
where $U_{1}$ and $U_{2}$ are two copies of $U$. 
The standard basis of $U$ consists of isotropic vectors $e,f$ with $(e.f)=1$. 
We denote the standard basis of $U_i$ by $e_i,f_i$, $i=1,2$,  
and denote by $h^2$ the element $(1,1,1)\in I_{3,0}\subset L$.

We will use the following theorem, an interesting result for itself, to prove Theorem \ref{mainthm1}.

\begin{thm}\label{mainprop}
Any two Hassett divisors intersect, 
i.e., $\mathcal{C}_{d_1}\cap \mathcal{C}_{d_2} \neq \emptyset$ 
for any two integers $d_{1}$ and $d_{2}$ satisfying $(\star)$. 
Moreover, there exists a smooth cubic fourfold $X$ and a codimesion-two subvariety $\mathcal{C}_{A(X)} \subset \mathcal{C}$ such that 
$X\in \mathcal{C}_{A(X)} \subset \mathcal{C}_{d_1}\cap \mathcal{C}_{d_2}$ 
and $A(X)$ is a rank 3 lattice of determinant $\displaystyle \frac{d_1d_2-1}{3}$ if $d_{1}\equiv 2\; (\mathrm{mod}\; 6)$ and $d_{2}\equiv 2\; (\mathrm{mod}\; 6)$; 
$\displaystyle \frac{d_1d_2}{3}$ otherwise.
\end{thm}

\begin{proof}

By definition, an integer $d$ satisfies $(\star)$ if $d>6$ and $d\equiv 0, 2\; (\mathrm{mod}\; 6)$.
Therefore, the proof is divided into three cases:

{\bf Case $(1)$: $d_{1}\equiv 0\; (\mathrm{mod}\; 6)$ and $d_{2}\equiv 0\; (\mathrm{mod}\; 6)$.}  
Suppose $d_{1}=6n_{1}$, $d_{2}=6n_{2}$ and $n_{1}, n_{2}\geq 2$.
We consider the rank $3$ lattice 
$$
h^{2}\in M:=\langle h^{2}, \alpha_{1}, \alpha_{2} \rangle \subset L
$$ 
generated by $h^{2}$, $\alpha_{1}:=e_{1}+n_{1}f_{1}$ and $\alpha_{2}:=e_{2}+n_{2}f_{2}$.
Then the Gram matrix of $M$ with respect to the basis $(h^2,\alpha_1,\alpha_2)$ is
$$
\left(\begin{array}{ccc} 
3 & 0 & 0 \\
0  & 2n_{1}  & 0 \\
0  & 0 & 2n_{2}
\end{array} \right).
$$
Therefore, $h^{2}\in M\subset L$ is positive definite saturated sublattice.
In addition, for any nonzero $v=xh^{2}+y\alpha_{1}+z\alpha_{2}\in M$, where $x,y,z$ are integers,
we have 
$$
(v.v)=3x^{2}+2n_{1}y^{2}+2n_{2}z^{2} \geq 3
$$
since $n_{1}, n_{2}\geq 2$ and at least one of the integers $x, y, z$ is nonzero. 
Hence, the embedding $h^2\in M \subset L$ satisfies the conditions of Lemma \ref{key-observation}. Thus, by Proposition \ref{crucial}, $\mathcal{C}_{M}\subset \mathcal{C}$ is nonempty and has codimension $2$, and there exists $X\in  \mathcal{C}_{M}$ with $A(X)=M$. Thus $A(X)$ is a rank $3$ lattice of $\displaystyle \det(A(X))=\frac{d_1d_2}{3}$. 
Moreover, we consider the sublattices
$$
h^{2}\in K_{d_{1}}:=\langle h^{2}, \alpha_{1}\rangle \subset M  
$$
with discriminant $d_{1}$, and 
$$
h^{2}\in K_{d_{2}}:=\langle h^{2},  \alpha_{2} \rangle \subset M  
$$
with discriminant $d_{2}$. Clearly, both $K_{d_{1}}$ and $K_{d_{2}}$ are saturated sublattices of $M$.
Applying Lemma \ref{key-observation}  and Proposition \ref{crucial} again, 
we obtain $X\in \mathcal{C}_{M}\subset \mathcal{C}_{K_{d_{1}}} = \mathcal{C}_{d_{1}}$ 
and $X\in \mathcal{C}_{M}\subset \mathcal{C}_{K_{d_{2}}} = \mathcal{C}_{d_{2}}$. 
Consequently, 
$X\in \mathcal{C}_{M}\subset \mathcal{C}_{d_{1}} \cap \mathcal{C}_{d_{2}}$ is what we want.

{\bf Case $(2)$: $d_{1}\equiv 0\; (\mathrm{mod}\; 6)$ and $d_{2}\equiv 2\; (\mathrm{mod}\; 6)$.}
Given $d_{1}=6n_{1}$ and $d_{2}=6n_{2}+2$ with $n_{1}\geq 2, n_{2}\geq 1$.
We consider the rank $3$ lattice 
$$
h^{2}\in M:=\langle h^{2}, \alpha_{1}, \alpha_{2}+(0, 0, 1) \rangle \subset L
$$ 
where $(0, 0, 1)\in I_{3, 0}$.
Then the Gram matrix of $M$ with respect to the basis $( h^{2}, \alpha_{1}, \alpha_{2}+(0, 0, 1))$ is
$$
\left(\begin{array}{ccc} 
3 & 0 & 1 \\
0  & 2n_{1}  & 0 \\
1  & 0 & 2n_{2}+1
\end{array} \right)
$$
Thus, $h^{2}\in M\subset L$ is positive definite saturated sublattice.
Futhermore, for any nonzero $v=xh^{2}+y\alpha_{1}+z(\alpha_{2}+(0,0,1))\in M$,
we get
$$
(v.v)=2x^{2}+2n_{1}y^{2}+2n_{2}z^{2}+(x+z)^{2}  \geq 3
$$
since $n_{1}\geq 2$,  $n_{2}\geq 1$ and at least one of the integers $x, y, z$ is nonzero. 
Hence, by Lemma \ref{key-observation} and Proposition \ref{crucial},  we conclude that $\mathcal{C}_{M}\subset \mathcal{C}$ is nonempty and has codimension $2$, and there exists $X\in  \mathcal{C}_{M}$ with $A(X)=M$. 
Thus $A(X)$ is a rank $3$ lattice of $\displaystyle \det(A(X))=\frac{d_1d_2}{3}$. 
Similarly, we consider the sublattices:
$$
h^{2}\in K_{d_{1}}:=\langle h^{2}, \alpha_{1}\rangle \subset M  
$$
of discriminant $d_{1}$, and
$$
h^{2}\in K_{d_{2}}:=\langle h^{2},  \alpha_{2} +(0, 0, 1)\rangle \subset M  
$$
of discriminant $d_{2}$.
Again Lemma \ref{key-observation} and Proposition \ref{crucial} imply $X\in \mathcal{C}_{M}\subset \mathcal{C}_{K_{d_{1}}} = \mathcal{C}_{d_{1}}$ 
and $X\in \mathcal{C}_{M}\subset \mathcal{C}_{K_{d_{2}}} = \mathcal{C}_{d_{2}}$. 
Consequently, 
$X\in \mathcal{C}_{M}\subset \mathcal{C}_{d_{1}} \cap \mathcal{C}_{d_{2}}$ 
is what we wanted.

{\bf Case $(3)$: $d_{1}\equiv 2\; (\mathrm{mod}\; 6)$ and $d_{2}\equiv 2\; (\mathrm{mod}\; 6)$.}
Assume $d_{1}=6n_{1}+2$ and $d_{2}=6n_{2}+2$ with $n_{1}, n_{2}\geq 1$.
We consider the rank $3$ lattice 
$$
h^{2}\in M:=\langle h^{2}, \alpha_{1}+(0,1,0), \alpha_{2}+(0,0,1) \rangle \subset L
$$ 
here $(0,1,0)\in I_{3,0}$.
Then the Gram matrix of $M$ with respect the basis is
$$
\left(\begin{array}{ccc} 
3 & 1 & 1 \\
1  & 2n_{1}+1  & 0 \\
1 & 0 & 2n_{2}+1
\end{array} \right)
$$
Thus, $h^{2}\in M\subset L$ is positive definite saturated sublattice.
For any nonzero $v=xh^{2}+y(\alpha_{1}+(0,1,0))+z(\alpha_{2}+(0,0,1))\in M$,
we obtain
$$
(v.v)=x^{2}+2n_{1}y^{2}+2n_{2}z^{2}+(x+y)^{2}+(x+z)^{2} \geq 3
$$
since $n_{1}, n_{2}\geq 1$ and at least one of the integers $x, y, z$ is nonzero. 
Hence, Lemma \ref{key-observation} and Proposition \ref{crucial} concludes that $\mathcal{C}_{M}\subset \mathcal{C}$ is nonempty and has codimension $2$, and there exists $X\in  \mathcal{C}_{M}$ with $A(X)=M$. 
Thus $A(X)$ is a rank $3$ lattice of $\displaystyle \det(A(X))=\frac{d_1d_2-1}{3}$. 
Moreover, we consider 
$$
h^{2}\in K_{d_{1}}:=\langle h^{2}, \alpha_{1}+(0,1, 0)\rangle \subset M  
$$
with discriminant $d_{1}$ and
$$
h^{2}\in K_{d_{2}}:=\langle h^{2},  \alpha_{2} +(0, 0, 1)\rangle \subset M  
$$
with discriminant $d_{1}$.
By Lemma \ref{key-observation} and Proposition \ref{crucial}, 
we obtain $X\in \mathcal{C}_{M}\subset \mathcal{C}_{K_{d_{1}}} = \mathcal{C}_{d_{1}}$ 
and $X\in \mathcal{C}_{M}\subset \mathcal{C}_{K_{d_{2}}} = \mathcal{C}_{d_{2}}$. 
As a consequence, 
$X\in \mathcal{C}_{M}\subset \mathcal{C}_{d_{1}} \cap \mathcal{C}_{d_{2}}$ 
is what we wanted.
This finishs the proof of Theorem \ref{mainprop}.
\end{proof}

\begin{rem}
Note that it has been known $\mathcal{C}_{8}\cap \mathcal{C}_{14}\neq \emptyset$ for $20$ years (Hassett \cite{Has99}), and $\mathcal{C}_{8}$ intersects every Hassett divisor (Addington-Thomas \cite[Theorem 4.1]{AT14}).
It is also shown that $\mathcal{C}_{8}\cap \mathcal{C}_{14}$ has five irreducible components (\cite{ABBVA14, BRS19}).
Moreover, \cite[page 166, paragraph 4 line 2 ]{BRS19} has mentioned that $\mathcal{C}_{14}$ intersects many other divisors $\mathcal{C}_{d}$, 
however, it is not obvious to see that which Hassett divisors intersect with $\mathcal{C}_{14}$. 
\end{rem}

Consequently, Theorem \ref{mainprop} not only generalizes \cite[Theorem 4.1]{AT14} but also implies that $\mathcal{C}_{14}$ intersects all Hassett divisors. 
Because of the same reason, we may conclude the main result of the current paper.

\begin{thm}[=Theorem \ref{mainthm}]\label{mainthm1}
Every Hassett divisor $\mathcal{C}_{d}$ contains a union of three codimension-two subvarieties in $\mathcal{C}$ parametrizing rational cubic fourfolds.
\end{thm}

\begin{proof}
Applying Theorem \ref{mainprop} to the pairs of integers $(d_1,d_2)=(d,14)$, $(d,26)$, $(d,38)$.
Then there exist three smooth cubic fourfolds $X_{1}$, $X_{2}$ and $X_{3}$
such that
$$
X_{1}\in \mathcal{C}_{A(X_{1})}\subset \mathcal{C}_{d}\cap \mathcal{C}_{14} \subset \mathcal{C}_{d},
$$
$$
X_{2}\in \mathcal{C}_{A(X_{2})}\subset \mathcal{C}_{d}\cap \mathcal{C}_{26} \subset \mathcal{C}_{d},
$$
$$
X_{3}\in \mathcal{C}_{A(X_{3})}\subset \mathcal{C}_{d}\cap \mathcal{C}_{38} \subset \mathcal{C}_{d},
$$
where 
$\mathcal{C}_{A(X_{1})}$, $\mathcal{C}_{A(X_{2})}$, and $\mathcal{C}_{A(X_{3})}$ 
are codimension-two subvarieties of $\mathcal{C}$.
Here $A(X_{1})$, $A(X_{2})$ and $A(X_{3})$ are three different rank $3$ lattices of determinants:
\begin{itemize}
\item if $d\equiv 0\; (\mathrm{mod}\; 6)$, 
         then $\displaystyle \det(A(X_{1}))=\frac{14d}{3}$, 
         $\displaystyle \det(A(X_{2}))=\frac{26d}{3}$ 
         and $\displaystyle \det(A(X_{3}))=\frac{38d}{3}$;
\item if $d\equiv 2\; (\mathrm{mod}\; 6)$, 
         then $\displaystyle \det(A(X_{1}))=\frac{14d-1}{3}$, 
         $\displaystyle \det(A(X_{2}))=\frac{26d-1}{3}$ 
         and $\displaystyle \det(A(X_{3}))=\frac{38d-1}{3}$.
\end{itemize}
By definition of $\mathcal{C}_{A(X_{i})}$ (see Proposition \ref{crucial}),
a smooth cubic fourfold $X\in \mathcal{C}_{A(X_{i})}$
only if there exists a saturated embedding $A(X_{i})\subset A(X)$.
Since $A(X_{1})$, $A(X_{2})$ and $A(X_{3})$ are rank $3$ lattices of different determinants, it follows that there is no saturated embedding $A(X_{i})\subset A(X_{j})$ if $i\neq j$.
Therefore,
$X_{i} \notin\mathcal{C}_{A(X_{j})}$ if $i\neq j$ and 
$\mathcal{C}_{A(X_{1})}$, $\mathcal{C}_{A(X_{2})}$, and $\mathcal{C}_{A(X_{3})}$ are three different codimension-two subvarieties.

Moreover, since every smooth cubic fourfold in $\mathcal{C}_{14}$, $\mathcal{C}_{26}$ and $\mathcal{C}_{38}$ is rational (\cite{BRS19, RS19}),
so every smooth cubic fourfold in $\mathcal{C}_{A(X_{1})}$,  
$\mathcal{C}_{A(X_{2})}$ and $\mathcal{C}_{A(X_{3})}$ is rational. 
Therefore, $\mathcal{C}_{A(X_{1})}$,  
$\mathcal{C}_{A(X_{2})}$ and $\mathcal{C}_{A(X_{3})}$ are three different codimension-two subvarieties which parametrize rational cubic fourfolds. The completes the proof of Theorem \ref{mainthm1}.
\end{proof}

Our main result also motivates the following natural question:

\begin{quest}\label{quset-conj}
Does every Hassett divisor $\mathcal{C}_{d}$ contains a union of countably infinite codimension-two subvarieties parametrizes rational cubic fourfolds for non-admissible value $d$? 
\end{quest}

The answer to Question \ref{quset-conj} has already known for $\mathcal{C}_{8}$ and $\mathcal{C}_{18}$ (\cite{Has99, AHTVA16}).

\begin{cor}
The answer to Question \ref{quset-conj} is yes 
if Conjecture \ref{mainconj} holds.
\end{cor}

Return to Conjecture \ref{mainconj}, as by-product of Theorem \ref{mainthm1} (=Theorem \ref{mainthm}), 
we have the following. 

\begin{cor}
For every admissible value $d$, 
the Hassett divisor $\mathcal{C}_{d}$ contains a union of three codimension-two subvarieties
in $\mathcal{C}$ parametrizing rational cubic fourfolds.
\end{cor}

To obtain more information about the Hassett divisors, 
it is of importance to notice that Addington-Thomas \cite[Theorem 4.1]{AT14} showed that for any $d$ satisfies $(\star)$ there exists a cubic fourfold $X\in \mathcal{C}_{8}\cap \mathcal{C}_{d}$ such that $X\in \mathcal{C}_{d'}$ for some admissible value $d'$.
Even if it is conjectured to be rational, however,
it is still unknown whether such a $X$ is rational or not. 
Using the idea of the proof of Theorem \ref{mainprop} and Theorem \ref{mainthm1}, 
we obtain a generalization of Addington-Thomas's result.

\begin{thm}\label{mainthm2}
If $d_{1}$ and $d_{2}$ satisfy the condition $(\star)$, 
then $\mathcal{C}_{14}\cap \mathcal{C}_{d_{1}}\cap \mathcal{C}_{d_{2}}$ contains a codimension-three subvariety in $\mathcal{C}$ parametrizing rational cubic fourfolds.
\end{thm}

\begin{proof}
The proof is divided into three cases: 

{\bf Case $(1)$:}  Given $d_{1}=6n_{1}$ and $d_{2}=6n_{2}$ with $n_{1}, n_{2}\geq 2$.
We consider the rank $4$ lattice 
$$
h^{2}\in M:=\langle h^{2}, \nu, \alpha_{1}, \alpha_{2} \rangle \subset L
$$ 
where $\nu=(3,1,0)\in I_{3,0}\subset L$, 
$\alpha_{1}:=e_{1}+n_{1}f_{1}$ and $\alpha_{2}:=e_{2}+n_{2}f_{2}$.
Then the Gram matrix of $M$ with respect to the basis $(h^{2}, \nu, \alpha_{1}, \alpha_{2})$ is
$$
\left(\begin{array}{cccc} 
3 & 4 & 0 & 0\\
4  & 10 & 0 & 0\\
0  & 0& 2n_{1} & 0\\
0  & 0 & 0& 2n_{2}
\end{array} \right)
$$
Thus, $h^{2}\in M\subset L$ is positive definite saturated sublattice.
For any nonzero $v=x_{1}h^{2}+x_{2}\nu+x_{3}\alpha_{1}+x_{4}\alpha_{2}\in M$,
we have 
$$
(v.v)=2(x_{1}+2x_{2})^{2}+x_{1}^{2}+2x_{2}^{2}+2n_{1}x_{3}^{2}+2n_{2}x_{4}^{2}\geq 3
$$
since $n_{1}, n_{2}\geq 2$ and at least one of the integers $x_{i}$ is nonzero ($i=1, 2, 3, 4$). 
Hence, Lemma \ref{key-observation} and Proposition \ref{crucial} conclude that $\mathcal{C}_{M}$ is nonempty 
and has codimension $3$. 
In addition, we consider the lattices
$h^{2}\in K_{14}=\langle h^{2}, \nu\rangle$ and 
$h^{2}\in K_{d_{i}}:=\langle h^{2}, \alpha_{i}\rangle \subset M$
with discriminant $d_{i}$.
By Lemma \ref{key-observation} and Proposition \ref{crucial}, 
we obtain ${C}_{M}\subset \mathcal{C}_{K_{d_{1}}} = \mathcal{C}_{d_{1}}$ 
and also $\mathcal{C}_{M}\subset \mathcal{C}_{K_{d_{2}}} = \mathcal{C}_{d_{2}}$. 
Consequently, 
$\emptyset \neq \mathcal{C}_{M}\subset \mathcal{C}_{14}\cap \mathcal{C}_{d_{1}} \cap \mathcal{C}_{d_{2}}$ 
is what we wanted, since every cubic fourfold in $\mathcal{C}_{14}$ is rational.

Since {\bf Case $(2)$} and {\bf Case $(3)$} are the same as {\bf Case $(1)$},
we just give the main ingredients and left the details to the interested readers.

{\bf Case $(2)$:}
Given $d_{1}=6n_{1}$ and $d_{2}=6n_{2}+2$ with $n_{1}\geq 2, n_{2}\geq 1$.
We consider the rank $4$ lattice 
$$
h^{2}\in M:=\langle h^{2}, \nu, \alpha_{1}, \alpha_{2}+(0,0,1) \rangle \subset L
$$
and its sublattices $K_{14}=\langle h^{2}, \nu\rangle$,
$
K_{d_{1}}:=\langle h^{2}, \alpha_{1}\rangle \subset M  
$
and
$
K_{d_{2}}:=\langle h^{2},  \alpha_{2} +(0, 0, 1)\rangle \subset M  
$.

{\bf Case $(3)$:}
Given $d_{1}=6n_{1}+2$ and $d_{2}=6n_{2}+2$ with $n_{1}, n_{2}\geq 1$.
We consider the rank $4$ lattice 
$$
h^{2}\in M:=\langle h^{2}, \nu, \alpha_{1}+(0,1,0), \alpha_{2}+(0,0,1) \rangle \subset L
$$ 
and its sublattices $K_{14}=\langle h^{2}, \nu\rangle$,
$
K_{d_{1}}:=\langle h^{2}, \alpha_{1}+(0,1, 0)\rangle \subset M  
$
and
$
K_{d_{2}}:=\langle h^{2},  \alpha_{2} +(0, 0, 1)\rangle \subset M  
$.
\end{proof}


\end{document}